\documentclass[12pt]{amsart}

\usepackage[utf8]{inputenc}
\usepackage{amsfonts}
\usepackage{amsthm}
\usepackage{amssymb}
\usepackage{amsmath}
\usepackage{amscd}
\usepackage{latexsym,dsfont}
\usepackage{bbm}
\usepackage{mathrsfs}

\usepackage{times}
\usepackage{microtype}
\usepackage{setspace}
\usepackage[margin=1.1in]{geometry}

\usepackage{cite}

\usepackage[colorlinks=true, pdfstartview=FitV, linkcolor=blue,
citecolor=blue, urlcolor=blue]{hyperref}
\usepackage{xcolor}

\newtheorem{theorem}{Theorem}[section]
\newtheorem{lemma}[theorem]{Lemma}

\newtheorem{cor}[theorem]{Corollary}

\newtheorem{defn}[theorem]{Definition}

\theoremstyle{definition}

\theoremstyle{remark}
\newtheorem{remark}[theorem]{Remark}

\numberwithin{equation}{section}
 \allowdisplaybreaks

\def\XXint#1#2#3{{\setbox0=\hbox{$#1{#2#3}{\int}$}
     \vcenter{\hbox{$#2#3$}}\kern-.5\wd0}}

\DeclareMathOperator{\Div}{div}

\newcommand{\op}{{\mathrm{op}}}

\DeclareMathOperator{\lip}{\mathrm{Lip}}

\newcommand{\R}{\mathbb{R}}
\newcommand{\rn}{{\mathbb{R}^n}}

\newcommand{\loc}{\mathrm{loc}}

\newcommand{\bphi}{\boldsymbol{\phi}}
\newcommand{\bpsi}{\boldsymbol{\psi}}

\newcommand{\vecf}{\mathbf f}
\newcommand{\vecg}{\mathbf g}

\newcommand{\vect}{\mathbf t}

\newcommand{\vecu}{\mathbf u}
\newcommand{\vecw}{\mathbf w}

\title[Solutions to degenerate $p$-Poisson Equations]
{Sobolev Inequalities and the Existence of Solutions to Degenerate $p$-Poisson Equations}

\author[Cruz Uribe, Dal, Rodney] {David Cruz-Uribe OFS,  Feyza Elif Dal, Scott Rodney}

\address{David Cruz-Uribe, OFS \\
Dept. of Mathematics \\
University of Alabama \\
 Tuscaloosa, AL 35487, USA}

\email{dcruzuribe@ua.edu}

\address{Feyza Elif Dal \\
Dept. of Mathematics \\
Y\i ld\i z Technical University \\
Esenler, Istanbul, 34220 Davutpasha, Turkey
}
\email{feyzadal@hotmail.com}

\address{Scott Rodney\\
Dept. of Mathematics, Physics and Geology \\ 
Cape Breton University \\
Sydney, NS B1Y3V3, CA} 

\email{scott\_rodney@cbu.ca}

\thanks{The first author is partially supported by a Simons Foundation
  Travel Support for Mathematicians Grant and by NSF Grant DMS-2349550.
  The second author is supported by the TUBITAK 2211-E Domestic Direct Doctorate Scholarship Program and 2214-A International Research Fellowship Programme for PhD Students. The third  author is partially supported by Natural Sciences and Engineering Research Council of Canada grants DG 2026-06169 and DDG 2024-00029.  This project is supported by  TUBITAK, the Scientific and Technological Research Council of T\"urkiye through a 2501 Joint Research Program grant 223N112.}



\begin{document}

\begin{abstract}
    In this paper we study an equivalence between the existence of a Sobolev inequality without gain,
    \[\|\varphi\|_{L^p(v,\Omega)} \leq S(p,1) \| \sqrt{Q}\nabla \varphi\|_{L^p(\Omega)},\]
    that holds for smooth functions of compact support and the existence of a degenerate weak solution $(u,\nabla u)\in QH^{1,p}_0(v,\Omega)$ to a Dirichlet problem for the $p$-Laplacian with a zero order term: 
    \begin{align*}
\left\{\begin{array}{rclr}
-v^{-1}\Div(|\sqrt{Q}\nabla u|^{p-2}Q\nabla u)+F|u|^{p-2}u&= &|f|^{p-2}f - v^{-1}\Div(v|g|^{p-2}g\vect), & x \in \Omega, \\
 u &=&0, & x \in \partial \Omega,
 \end{array}
 \right.
\end{align*}
More precisely, we use the Sobolev inequality to prove the existence of a degenerate weak solution to this equation and then use the existence of such a solution to produce a Sobolev inequality.  Moreover, we show that solutions are unique.
\end{abstract}

\maketitle

\section{Introduction}
\label{section:intro}

In this paper we give a characterization of when degenerate weak solutions to the Dirichlet problem for the  $p$-Laplacian,
\begin{equation}\label{eqn:dirichlet-intro}
\begin{cases}
-v^{-1}\Div(|\sqrt{Q}\nabla u|^{p-2}Q\nabla u)+F|u|^{p-2}u = |f|^{p-2}f - v^{-1}\Div(v|g|^{p-2}g\vect), & x \in \Omega, \\
 \hspace{2.65in} u =0, & x \in \partial \Omega,
 \end{cases}
\end{equation}
exist and are unique.  Here $v$ is a weight (that is, a non-negative, measurable function) and $Q$ is an $n\times n$, symmetric, positive semidefinite, measurable matrix function.  (For brevity we defer precise assumptions on the coefficients and data to below.)  We show that solutions exist and are unique if and only if there exists a global, degenerate Sobolev inequality without gain that holds for all smooth $\varphi$ with compact support:
\[ \|\varphi\|_{L^p(v,\Omega)} \leq S(p,1) \| \sqrt{Q}\nabla \varphi\|_{L^p(\Omega)}.  \]

\medskip

Before stating our precise theorem, we give some previous results for context.  If $v=1$ and $Q$ is the $n\times n$ identity matrix, then this is the classical $p$-Laplacian; the existence of solutions to this equation is well-known and follows from a general theorem of Lions~\cite[Theorem~2.1, p.~171]{MR259693}.  The degenerate equation was first considered in the linear case (i.e., when $p=2$) by Fabes, Kenig, and Serapioni~\cite{MR643158}, and for all $1<p<\infty$ by Modica~\cite{MR839035}.  They considered the simpler equation when $F=g=0$ and they assumed that $v$ was in the Muckenhoupt class $A_p$, and the matrix satisfied $Q=v\tilde{Q}$, where $\tilde{Q}$ was uniformly elliptic.  

The degenerate equation was next considered when the matrix $Q$ satisfies the two weight ellipticity condition
\begin{equation} \label{eqn:degen-elliptic}
    w(x) |\xi|^p \leq |\sqrt{Q(x)}\xi|^p \leq v(x) |\xi|^p, \quad \xi \in \R^n, \quad \text{ a.e. } x \in \Omega, 
\end{equation} 
where $v$ is doubling, $w\in A_p$, and there exists $\sigma>1$ such that $v$ and $w$ satisfy the balance condition: 
\[  \frac{r_1}{r_2} \bigg( \frac{v(B(x_1,r_1)}{v(B(x_2,r_2)}\bigg)^{\frac{1}{\sigma p}} 
\leq c \bigg( \frac{w(B(x_1,r_1)}{w(B(x_2,r_2)}\bigg)^{\frac{1}{p}}, \]
where $0<r_1<r_2$, $x_1,\,x_2\in \Omega$, and $B(x_1,r_1)\subset B(x_2,r_2)$. Regularity results for solutions to the homogeneous equation without a zero order term (i.e., $F=f=g=0$) were proved when $p=2$ by Chanillo and Wheeden~\cite{MR847996} and for $1<p<\infty$ by Ferrari~\cite{MR2228656}.  The existence of weak solutions was shown by the first author, Moen and Naibo~\cite{MR3011287}, using an abstract result from Kinderlehrer and Stampacchia~\cite{MR1786735}.

Regularity (boundedness, Harnack inequality, continuity) for the degenerate equation when $v=1$ was considered by the third author, Monticelii, and Wheeden~\cite{Monticelli-Rodney-Wheeden,MR2906551} as a special case of a more general quasilinear operator.  Existence of solutions to this equation when $f=g=0$ was shown by Chua and Wheeden~\cite{MR3710683} as a special case of a more general result.  More recently, existence and uniqueness for this equation when $p=2$ was proved by the authors, \c{C}etin, and Zeren~\cite[Theorem~3.6]{SC-DCU-FED-SR-Zeren}, and the boundedness of solutions to this equation (when $F=0$) was studied by the first and third authors and MacDonald~\cite{Cruz-Uribe-Rodney,Cruz-Uribe-MacDonald-Rodney}.

Finally, in~\cite{Cruz-Uribe-Rodney-Rosta} the first and third authors and Rosta showed that the existence of a solution to a related Neumann-type problem (with $F=g=0$ in the equation) was equivalent to the existence of a degenerate Poincar\'e inequality.

\medskip

We now state our main results.  We first state some basic assumptions; however, we will defer some more technical definitions to Section~\ref{section:prelim} below.  Throughout, $n$ will be the dimension of the underlying space $\R^n$.  The set $\Omega \subset \R^n$ will always be a bounded, connected, open set.  The weight $v$ will be non-negative and measurable, but is not necessarily even locally integrable.  The value $p$ will be fixed, $1<p<\infty$.  The matrix $Q$ will be a symmetric, positive semidefinite, measurable matrix function such that
\[ |Q|_\op = \sup_{|\xi|=1} |Q\xi| \in L^{\frac{p}{2}}_\loc(\Omega).  \]

\begin{defn}
    Given $1 <p<\infty$, the pair $(v,Q)$ is said to have the  $p$-Sobolev property  on $\Omega$ if there is a positive constant $S(p,1)$ such that for all $\varphi\in Q\lip_0(v,\Omega)$,
\begin{equation}\label{Sobolev prop}
    \bigg(\int_\Omega|\varphi|^{ p}\,vdx\bigg)^{\frac{1}{p}}
    \leq S(p,1) \bigg(\int_\Omega |\sqrt{Q}\nabla \varphi |^p \,dx\bigg)^{\frac{1}{p}}.
\end{equation}
    
\end{defn}

\begin{remark}
    As we will show below (see Lemma~\ref{lemma:approximation}), by a standard approximation argument, if $(v,Q)$ has the $p$-Sobolev property, then inequality~\eqref{Sobolev prop} holds for all $(u,\nabla u) \in QH_0^{1,p}(v,\Omega)$. 
\end{remark}

\medskip

Going forward, in our Dirichlet problem we will always  make the following assumptions:
\begin{itemize}
    \item the data functions $f,\,g \in L^p(v,\Omega)$;

    \item $F\in L^\infty(v,\Omega)$ and $F$ is non-negative;

    \item $\vect : \Omega \to \R^n$ is a degenerate subunit vector-field, i.e.,
    \begin{equation} \label{eqn:subunit}
    \|\vect \cdot \nabla \varphi \|_{L^p(v,\Omega)} \leq C(\vect) \|\nabla \varphi\|_{QL^p(\Omega)}. 
    \end{equation}
\end{itemize}

\begin{remark}
    The subuniticity of $\vect$ follows from the stronger pointwise inequality
    \[ |\vect(x)\cdot \xi| \leq \frac{1}{\sqrt{v(x)}} |\sqrt{Q(x)}\xi|, \quad \xi \in \R^n, \quad \text{ a.e. } x\in \Omega. \]
    The pointwise subuniticity condition when $v=1$ was introduced in~\cite{sawyer2006holder}; the more general condition for all $v$ was introduced in~\cite{SC-DCU-FED-SR-Zeren}.
\end{remark}

\begin{defn}
    Given $1< p<\infty$, the pair $(v,Q)$ is said to have the $p$-Dirichlet property on $\Omega$ if 
        for any $F$, $\vect$, $f$, $g$  as above, there exists a degenerate weak solution $(u,\nabla u)\in QH^{1,p}_0(v,\Omega)$ to the 
        Dirichlet problem
\begin{equation}\label{Dirichlet prob.}
\begin{cases}
-v^{-1}\Div(|\sqrt{Q}\nabla u|^{p-2}Q\nabla u)+F|u|^{p-2}u = |f|^{p-2}f - v^{-1}\Div(v|g|^{p-2}g\vect), & x \in \Omega, \\
 u =0, & x \in \partial \Omega. 
 \end{cases}
\end{equation}
Moreover, any degenerate weak solution $(u,\nabla u)\in QH^{1,p}_0(v,\Omega)$ of \eqref{Dirichlet prob.} is regular: that is, there is a positive constant $c_0$ such that
\begin{equation}\label{regularity}
    \|u\|_{L^p(v,\Omega)}\leq c_0\big[ \|f\|_{L^p(v,\Omega)} + \|g\|_{L^p(v,\Omega)} \big].
\end{equation}
\end{defn}

Our main result shows that  these two properties are equivalent.

\begin{theorem}\label{equivalence}
    Given $1<p<\infty$, the pair $(v,Q)$ has the $p$-Dirichlet property on $\Omega$ if and only if it has the $p$-Sobolev property  on $\Omega$.  Furthermore, in this case solutions to~\eqref{Dirichlet prob.} are unique.
\end{theorem}

As a corollary to the proof we get that solutions satisfy a Caccioppoli-type inequality. See Lemma~\ref{Lemma5.1} and inequality~\eqref{eqn:grad-regularity2} below.

\begin{cor} 
    Given $1<p<\infty$, if the pair $(v,Q)$ has the $p$-Sobolev and $p$-Dirichlet properties, then every degenerate weak solution $(u,\nabla u)$ of the Dirichlet problem~\eqref{Dirichlet prob.} satisfies 
    \begin{equation} \label{eqn:grad-regularity}
    \|\nabla u\|_{QL^p(\Omega)}\leq c_1\big[ \|f\|_{L^p(v,\Omega)} + \|g\|_{L^p(v,\Omega)} \big]
\end{equation}
where $c_1$ is a positive constant independent of $(u,\nabla u),~f,~$and $g$.
\end{cor}


\begin{remark}
    The implication that the $p$-Sobolev property implies the $p$-Dirichlet property in Theorem~\ref{equivalence} is a generalization of~\cite[Theorem~3.6]{SC-DCU-FED-SR-Zeren}, which proves existence and uniqueness when $p=2$.  The necessity of the Sobolev inequality , however, is new even in this case.  We note that Sawyer and Wheeden studied the necessity of the Sobolev inequality with gain for what they called $L^q$ subellipticity for the homogeneous Dirichlet problem.  This required the weak solution to satisfy an interior $C^\alpha$ estimate, a stronger condition than \eqref{regularity}. Their work is related to ours when $v=1$ and $p=2$.  We refer the reader to \cite[Chapter 1.1.1]{sawyer2006holder} and \cite[Lemma 75]{sawyer2006holder} for more details.
\end{remark}

\begin{remark}
    Building on the work of Fabes, Kenig and Serapioni~\cite{MR643158} and Chanillo and Wheeden~\cite{MR805809}, it is possible to give examples of pairs $(v,Q)$ such that the $p$-Sobolev property holds.  (This is done for the Poincar\'e inequality in~\cite{Cruz-Uribe-Rodney-Rosta}, but essentially the same arguments work for the Sobolev inequality.)  Furthermore, in a recent paper~\cite{CDR-preprint}, the authors have shown that if there exists a function $w$ such that  $(v,Q)$ satisfies the degenerate ellipticity condition~\eqref{eqn:degen-elliptic}, then inequality~\eqref{Sobolev prop} holds provided that for some $q$, $1\leq q\leq p$, and some $\sigma$, $ 1\leq \sigma \leq \frac{n}{n-q}$ if $q<n$, and $\sigma\geq 1$ if $q\geq n$, $v\in L^{\frac{\sigma q}{\sigma q-p}}(\Omega)$ and $w^{-1} \in L^{\frac{q}{q-p}}(\Omega)$.  (Here we interpret $L^{\frac{a}{0}}(\Omega)$ as $L^\infty(\Omega)$.)  We stress that the only assumptions on $v$ and $w$ are integrability conditions, and we do not have to assume either Muckenhoupt $A_p$ or doubling conditions.  In all of these examples, Theorem~\ref{equivalence} shows that the Dirichlet problem~\eqref{Dirichlet prob.} has a unique solution.
\end{remark}

\begin{remark}
In~\cite{Cruz-Uribe-MacDonald-Rodney}, with stronger hypotheses on $v$, $Q$, $f$, and $g$, it was shown that solutions to~\eqref{Dirichlet prob.} are exponentially integrable.  Moreover, if a global Sobolev inequality with gain (that is, the $L^p(v,\Omega)$ norm on the right-hand side of inequality~\eqref{Sobolev prop} is replaced by an $L^{\sigma p}(v,\Omega)$   norm, $\sigma>1$, or an Orlicz norm $L^A(v,\Omega)$, $A(t) = t^p\log(e+t)^\sigma$, $\sigma>1$, then the solutions are bounded.  Thus, Theorem~\ref{equivalence} shows that these results are not vacuous.
\end{remark}

\begin{remark}
 One feature of our assumptions that we want to highlight is that the weight $v$ is decoupled from the matrix $Q$.  In earlier work (cf.~\cite{MR643158,MR847996}), $v$ was assumed to be the largest eigenvalue of $Q$ and had some additional regularity. In~\cite{SC-DCU-FED-SR-Zeren}, in order to prove existence and uniqueness of the linear equation with first order terms, it was necessary to assume  $|Q(x)|_\op \leq Cv(x)$, and also to assume that $v\in L^1(\Omega)$.  However,  for the equation with only the second order term, they showed that these assumptions were not necessary. (See~\cite[Theorem~3.6]{SC-DCU-FED-SR-Zeren}.)
\end{remark}

\begin{remark}
    While in~\cite{Cruz-Uribe-Rodney-Rosta} the existence of solutions to the Neumann-type problem was proved, the authors did not prove uniqueness.  The argument given below in Section~\ref{section:unique} can be adapted to prove uniqueness in that paper.  Details are left to the interested reader.
\end{remark}

\begin{remark}
    In~\cite{MR4332462}, the first and third authors and Penrod extended the results in~\cite{Cruz-Uribe-Rodney-Rosta}  to a Neumann-type problem for a variable exponent, degenerate $p(\cdot)$-Laplacian equation of the form
    \[ \Div(|\sqrt{Q}\nabla u|^{p(\cdot)-2} Q\nabla u) = |f|^{p(\cdot)-2} f v^{p(\cdot)},  \]
    assuming a degenerate, variable exponent Poincar\'e inequality.  It is an open problem to extend Theorem~\ref{equivalence} to this setting. 
\end{remark}

\medskip

The remainder of this paper is organized as follows.  In Section~\ref{section:prelim} we define the weighted Lebesgue and Sobolev spaces that are the solution spaces for~\eqref{Dirichlet prob.}, and we give a precise definition of degenerate weak solutions to this Dirichlet problem.  In Section~\ref{section:dirichlet-to-Sobolev} we prove that the $p$-Dirichlet property implies the $p$-Sobolev property.  Our proof depends on a Caccioppoli-type inequality:  see Lemma~\ref{Lemma5.1}.  In Section~\ref{section:sobolev-to-dirichlet} we prove  the converse.  Our proof of existence uses Minty's theorem~\cite[Chapter~II, Proposition~2.2]{Showalter}.  Finally, in Section~\ref{section:unique} we prove the uniqueness of solutions by adapting an argument due to Yaremenko~\cite{Yaremenko}.

\section{Preliminaries}
\label{section:prelim}

In this section we define the degenerate Sobolev spaces we use and then use them to define weak solutions to the Dirichlet problem~\eqref{Dirichlet prob.}.  We follow the notation and development from~\cite{SC-DCU-FED-SR-Zeren} and we refer the reader there for further information.

\subsection*{Sobolev Spaces}
Given a measurable function $v$ and $1<p<\infty$, define the space $L^p(v,\Omega)$ to be the collection of all measurable functions $f$ such that
\[ \|f\|_{L^p(v,\Omega)} = \bigg(\int_\Omega |f|^p\,vdx \bigg)^{\frac{1}{p}} < \infty.  \]
It is well-known that $L^p(v,\Omega)$ is a separable, reflexive Banach space.  (See~\cite[Chapter~3]{MR924157}.)

Again for $1<p<\infty$, given a matrix function $Q$, let $QL^p(\Omega)$  be the space of measurable functions $\vecf : \Omega \rightarrow \rn$ such that
\[ \|\vecf\|_{QL^p(\Omega)} = \bigg(\int_\Omega |\sqrt{Q}\vecf|^p\,dx\bigg)^{\frac{1}{p}}
< \infty. \]
Since we have that $|Q|_\op \in L^{\frac{p}{2}}_\loc(\Omega)$, $QL^p(\Omega)$ is also a separable, reflexive Banach space.  See~\cite[Lemma~2.1]{SC-DCU-FED-SR-Zeren}.  It follows immediately that the direct sum $L^p(v,\Omega) \oplus QL^p(\Omega)$ is also a separable, reflexive Banach space.

\medskip

We can use these spaces to define our degenerate Sobolev space.  Let $\lip_\loc(\Omega)$ be the collection of functions that are Lipschitz on compact subsets $\Omega$, and let $\lip_0(\Omega)$ be those functions in $\lip_\loc(\Omega)$ which have compact support in $\Omega$.  Given
$1\leq p<\infty$, let $Q\lip(v,\Omega)$ be the collection of $\varphi \in \lip_\loc(\Omega)$ such that
\[ \|\varphi\|_{QH^{1,p}(v,\Omega)} 
= \|\varphi\|_{L^p(v,\Omega)} + \|\nabla \varphi \|_{QL^p(\Omega)}< \infty. 
\]
Define the set $Q\lip_0(v,\Omega)= \lip_0(\Omega)\cap Q\lip(v,\Omega) $.  Note that if $v\in L^1_\loc(\Omega)$, then $\lip_0(\Omega) = Q\lip_0(v,\Omega)$.

Define the degenerate Sobolev space $QH^{1,p}(v,\Omega)$ to be the abstract closure of $Q\lip(v,\Omega)$ with respect to this norm.  Formally, this space is the collection of equivalence classes of Cauchy sequences with respect to this norm in $Q\lip(v,\Omega)$.  However, we can identify it with a closed subspace  of the Banach space $L^p(v,\Omega) \oplus QL^p(\Omega)$.  Given an equivalence class in $QH^{1,p}(v,\Omega)$ represented by the Cauchy sequence $\{\varphi_k\}_{k=1}^\infty$, the sequence $\{(\varphi_k,\nabla \varphi_k)\}_{k=1}^\infty$ is Cauchy in $L^p(v,\Omega)\oplus QL^p(\Omega)$; since the direct sum of Banach spaces is again a Banach space, this sequence converges in $L^p(v,\Omega)\oplus QL^p(\Omega)$ to a  limit $(u,\vecg)$, and the collection of all such limits forms a closed subspace, and so is itself a Banach space.  Since this limit is unique to each equivalence class, we can identify each element of $QH^{1,p}(v,\Omega)$ with this ordered pair.  In an abuse of notation, we will denote $\vecg$ by $\nabla u$, and will refer to it as the degenerate weak gradient of $u$.  However, depending on the matrix $Q$, $\nabla u$ may not be a weak derivative in the sense of distributions.  Moreover, it may not even be uniquely determined by $u$.  We refer the reader to~\cite{SC-DCU-FED-SR-Zeren} for further information.

\begin{remark} \label{remark:grad-linear}
Even though the degenerate weak gradient is not a distributional derivative, it does have many of the properties of the classical definition of a weak derivative.  In particular, below we will need that the degenerate weak gradient is linear: if $u,\,v \in QH^{1,p}(v,\Omega)$, by looking at the Cauchy sequences of Lipschitz functions used to define them, we have that $\nabla(u-v) = \nabla u -\nabla v$.  For additional properties, see~\cite{Cruz-Uribe-Rodney}.
\end{remark}

We define the space $QH_0^{1,p}(v,\Omega)$ to be  the closure of $Q\lip_0(v,\Omega)$ in $QH^{1,p}(v,\Omega)$.   Since it is a closed subspace of the reflexive, separable Banach space $L^p(v,\Omega) \oplus QL^p(\Omega)$, $QH_0^{1,p}(v,\Omega)$ is itself a reflexive, separable Banach space.

\begin{lemma} \label{lemma:approximation}
If $(v,Q)$ have the $p$-Sobolev inequality, then inequality~\eqref{Sobolev prop} holds for any pair $(u,\nabla u) \in QH_0^{1,p}(v,\Omega)$.
\end{lemma}

\begin{proof}
    If $(u,\nabla u) \in QH_0^{1,p}(v,\Omega)$, then there exists a sequence $\{u_k\}_{k=1}^\infty$, such that $u_k \to u$ in $L^p(v,\Omega)$ and $\nabla u_k \to \nabla u$ in $QL^p(\Omega)$.  By passing to a subsequence, we may also assume that $u_k\to u$ pointwise $v$-almost everywhere.  The desired inequality follows at once by Fatou's lemma.
\end{proof}

\subsection*{Weak Solutions}
We can now define degenerate weak solutions to the Dirichlet problem~\eqref{Dirichlet prob.}. Given $f,\,g \in L^p(v,\Omega)$, we say that a pair $(u,\nabla u)\in QH^{1,p}_0(v,\Omega)$ is a degenerate weak solution 
if for all test functions $\varphi\in Q\lip_0(v,\Omega)$,
\begin{multline}\label{weak soln.}
    \int_\Omega\big|\sqrt{Q}\nabla u \big|^{p-2}Q\nabla u \cdot \nabla\varphi \, dx
    +\int_\Omega F|u|^{p-2}\, u\,\varphi \, vdx\\
    =\int_\Omega|f|^{p-2}f\varphi \, vdx
    + \int_\Omega |g|^{p-2}g \, \vect \cdot \nabla \varphi\, vdx.
\end{multline}

With our assumptions we have, {\em a priori},  that both sides of \eqref{weak soln.} are finite. Since $u\in L^p(v,\Omega)$, $\nabla u\in QL^p(\Omega)$ and $\varphi\in Q\lip_0(v,\Omega)$, by Hölder's inequality we get
\begin{align*}\label{ineq}
   &  \bigg|\int_\Omega\big|\sqrt{Q}\nabla u\big|^{p-2}Q\nabla u \cdot \nabla\varphi  \, dx\bigg|
    +\bigg|\int_\Omega F|u|^{p-2}\varphi\, u\, vdx\bigg|\\
& \qquad \qquad \leq 
\int_\Omega\big|\sqrt{Q}\nabla u \big|^{p-1}\big|\sqrt{Q}\nabla \varphi\big|\, dx
+\int_\Omega F|u|^{p-1}|\varphi|\, vdx\\
& \qquad \qquad \leq 
\|\nabla u \|^{p-1}_{QL^p(\Omega)}\|\nabla \varphi\|_{QL^p(\Omega)}+\|F\|_{L^\infty(v,\Omega)}\|u\|^{p-1}_{L^p(v,\Omega)}\|\varphi\|_{L^p(v,\Omega)} \\
& \qquad \qquad <\infty.
\end{align*}
Similarly, since $f,\, g \in L^p(v,\Omega)$, and $\vect$ is a degenerate subunit vector-field,
\begin{multline*}   \bigg|\int_\Omega|f|^{p-2}f\varphi\, vdx\bigg|
+ \bigg| \int_\Omega |g|^{p-2}g \vect \cdot \nabla \varphi\, vdx \bigg|\\
\leq \|f\|^{p-1}_{L^p(v,\Omega)}\| \varphi\|_{L^p(v,\Omega)}
+ \|g\|^{p-1}_{L^p(v,\Omega)}\|\vect \cdot \nabla \varphi\|_{L^p(v,\Omega)}
<\infty.  
\end{multline*}

The next result follows from the definition of a degenerate weak solution and an approximation argument using H\"older's inequality as in the above estimates.  For the case $p=2$, see~\cite[Lemma~4.5]{SC-DCU-FED-SR-Zeren}.

\begin{lemma}\label{test func.}
    If $(u,\nabla u)\in QH^{1,p}_0(v,\Omega)$ is a degenerate weak solution of the Dirichlet problem~\eqref{Dirichlet prob.}, then we can use $QH^{1,p}_0(v,\Omega)$ functions as test functions:  that is,~\eqref{weak soln.} also holds with $(\varphi,\nabla \varphi)$ replaced by any pair $(w,\nabla w)\in QH^{1,p}_0(v,\Omega)$. 
\end{lemma}

\section{The Proof that $p$-Dirichlet Implies $p$-Sobolev}
\label{section:dirichlet-to-Sobolev}

In this section we prove the first half of Theorem~\ref{equivalence} by showing that the existence of solutions to the Dirichlet problem~\eqref{Dirichlet prob.} implies that the Sobolev inequality~\eqref{Sobolev prop} holds.  To prove this we need to prove the following Caccioppoli-type  inequality.

\begin{lemma}\label{Lemma5.1}
    Given $1<p<\infty$ and $f,\,g \in L^p(v,\Omega)$, suppose  $(u,\nabla u)\in QH^{1,p}_0(v,\Omega)$ is a degenerate weak solution of the Dirichlet problem~\eqref{Dirichlet prob.} that satisfies~\eqref{regularity}.  Then there is a constant $c_1>0$ so that $\nabla u \in QL^p(\Omega)$ satisfies 
    \begin{equation} \label{eqn:caccioppoli}
    \|\nabla u\|_{QL^p(\Omega)}\leq c_1\big[ \|f\|_{L^p(v,\Omega)} + \|g\|_{L^p(v,\Omega)}\big].
    \end{equation}
\end{lemma}

\begin{proof}
    Let $(u,\nabla u)\in QH^{1,p}_0(v,\Omega)$ be a degenerate weak solution of \eqref{Dirichlet prob.} that satisfies~\eqref{regularity}. By Remark~\ref{test func.} we can take the pair $(u,\nabla u)$ as our test function in the definition~\eqref{weak soln.} to get
    \begin{align} \label{eqn:caccioppoli2}
        \|\nabla u\|^p_{QL^p(\Omega)}
        &=\int_\Omega\big|\sqrt{Q}\nabla u\big|^{p-2}Q\nabla u \cdot \nabla u \,dx \notag \\
        &= \int_\Omega|f|^{p-2}f uv\,dx
        + \int_\Omega |g|^{p-2}g \vect \cdot \nabla u\, vdx
        -\int_\Omega F|u|^{p-2}u^2 \,vdx \notag \\
        &\leq \int_\Omega|f|^{p-1}|u| \,vdx
        + \int_\Omega |g|^{p-1} |\vect \cdot \nabla u|\, vdx. \notag\\
\intertext{In the last inequality we used that $F$ is non-negative.  By Hölder's inequality, inequality~\eqref{eqn:subunit}, inequality \eqref{regularity}, and Young's inequality,}
    &\leq \|f\|^{p-1}_{L^p(v,\Omega)}\|u\|_{L^p(v,\Omega)}
    + \|g\|^{p-1}_{L^p(v,\Omega)}\|\vect \cdot \nabla u\|_{L^p(v,\Omega)}  \notag\\
    &\leq c_0\|f\|^{p-1}_{L^p(v,\Omega)}\big[\|f\|_{L^p(v,\Omega)} + \|g\|_{L^p(v,\Omega)}\big]
    + C(\vect) \|g\|^{p-1}_{L^p(v,\Omega)}\|\nabla u\|_{QL^p(\Omega)} \\
   &\leq c_1 \big[\|f\|^{p}_{L^p(v,\Omega)}
   + \|g\|^{p}_{L^p(v,\Omega)}\big] + \tfrac{1}{2} \|\nabla u\|_{QL^p(\Omega)}^p, \notag
\end{align}
where $c_1 = c_1(c_0,p,C(\vect))$.  If we now rearrange terms, we get inequality~\eqref{eqn:caccioppoli}.
\end{proof}

The proof that the $p$-Dirichlet property implies the $p$-Sobolev property follows directly from Lemma \ref{Lemma5.1}. If $f=0$ there is nothing to prove, so fix $f\in Q\lip_0(v,\Omega)$, $\|f\|_{L^{p}(v,\Omega)}> 0$. 
Let $F=g= 0$ and let $(u,\nabla u)\in QH^{1,p}_0(v,\Omega)$ be a weak solution of \eqref{Dirichlet prob.} corresponding to this function $f$. Since $f$ itself is a valid test function, by \eqref{weak soln.}, H\"older's inequality, and Lemma~\ref{Lemma5.1},  we get
\begin{multline*}
    \|f\|^p_{L^p(v,\Omega)}
    =\bigg|\int_\Omega|f|^{p-2}f \,f \,v dx\bigg|
    =\bigg|\int_\Omega\big|\sqrt{Q}\nabla u\big|^{p-2} Q\nabla u \cdot \nabla f\,dx\bigg|\\
    \leq \int_\Omega \big|\sqrt{Q}\nabla u\big|^{p-1}\big|\sqrt{Q}\nabla f\big|\,dx 
    \leq \|\nabla u\|^{p-1}_{QL^p(\Omega)}\|\nabla f\|_{QL^p(\Omega)}
  \leq c_1^{p-1} \|f\|^{p-1}_{L^p(v,\Omega)}\|\nabla f\|_{QL^p(\Omega)}.
\end{multline*}
Since $\|f\|^{p-1}_{L^p(v,\Omega)}> 0$ we can divide by this quantity to get the Sobolev inequality~\eqref{Sobolev prop}.  This completes the proof that the $p$-Dirichlet property implies the $p$-Sobolev property.

\section{The proof that $p$-Sobolev Implies $p$-Dirichlet}
\label{section:sobolev-to-dirichlet}

In this section we prove the second half of Theorem~\ref{equivalence} by showing that if the Sobolev inequality~\eqref{Sobolev prop} holds, then there exist degenerate weak solutions to the Dirichlet problem~\eqref{Dirichlet prob.}, and we prove that the solution satisfies~\eqref{regularity}.
 We will prove existence  by using Minty's theorem \cite[Chapter II, Proposition 2.2]{Showalter}, which can be thought of as a Banach space version of the  Lax-Milgram theorem. To state this result we first introduce some notation. Given a separable, reflexive Banach space $\mathcal{B}$, denote its dual space by $\mathcal{B}^*$. Given a functional $\Gamma \in \mathcal{B}^*$, write its value at $\varphi\in \mathcal{B}$ as $\langle \Gamma,\varphi\rangle$. Thus, if $\mathcal{T}:\mathcal{B}\rightarrow\mathcal{B}^*$ and $\vecu\in \mathcal{B}$, then we have $\mathcal{T}(\vecu)\in \mathcal{B}^*$ and so its value at $\varphi$ is denoted by $\langle\mathcal{T}(\vecu),\varphi\rangle$. We say that the operator $\mathcal{T}$ is bounded if given any bounded set $U\subset \mathcal{B}$, $\mathcal{T}(U)$ is bounded in $\mathcal{B}^*$.

\begin{theorem}\label{Minty's}
Let $\mathcal{B}$ be a separable, reflexive Banach space and fix $\Gamma\in \mathcal{B}^*$. Suppose that $\mathcal{T}:\mathcal{B}\rightarrow\mathcal{B}^*$ is a bounded operator that has the following properties:
\begin{enumerate}
    \item Monotone: $\langle\mathcal{T}(\vecu)-\mathcal{T}(\vecw),\vecu-\vecw\rangle\ge 0$ for all $\vecu,\vecw\in \mathcal{B}$;
    \item Hemicontinuous: for $z\in \R,$ the mapping $z\mapsto\langle\mathcal{T}(\vecu+z\vecw),\vecw\rangle$ is continuous for all $\vecu,\vecw\in \mathcal{B}$;
    \item Almost coercive: there exists a constant $\lambda>0$ so that $\langle\mathcal{T}(\vecu),\vecu\rangle\ge \langle\Gamma,\vecu\rangle$ for any $\vecu\in\mathcal{B}$ satisfying $\|\vecu\|_\mathcal{B}>\lambda$.
\end{enumerate}
  Then the set of $\vecu\in\mathcal{B}$ such that $\mathcal{T}(\vecu)=\Gamma$ is non-empty.  
\end{theorem}

We apply Minty's theorem to find a degenerate weak solution of~\eqref{Dirichlet prob.} as follows.  Let $\mathcal{B}=QH^{1,p}_0(v,\Omega)$. Given $\vecu=(u,\nabla u)$ and $\vecw=(w,\nabla w)$ in $QH^{1,p}_0(v,\Omega)$, first define the operator $\mathcal{T}:QH^{1,p}_0(v,\Omega)\rightarrow QH^{1,p}_0(v,\Omega)^*$ by
\begin{equation*}
    \langle\mathcal{T}(\vecu),\vecw\rangle
    =\int_\Omega\big|\sqrt{Q} \nabla u \big|^{p-2} Q\nabla u \cdot \nabla w\,dx
    + \int_\Omega F|u|^{p-2}u w \, vdx.
\end{equation*}
Now define the functional $\Gamma$ by
\begin{equation*}
    \langle\Gamma,\vecw\rangle=\int_\Omega|f|^{p-2}fw \,v dx
    + \int_\Omega |g|^{p-2}g\vect \cdot \nabla w \,vdx;
\end{equation*}
since $f,\,g\in L^p(v,\Omega)$, Hölder's inequality and the subunit condition \eqref{eqn:subunit} shows that $\Gamma$ is well defined on $QH^{1,p}(v,\Omega)$. Thus, $\vecu=(u,\nabla u)$ is a degenerate weak solution of \eqref{Dirichlet prob.} if and only if
\begin{equation*}
    \langle\mathcal{T}(\vecu),\vecw\rangle=\langle\Gamma,\vecw\rangle
\end{equation*}
for all $\vecw\in QH^{1,p}_0(v,\Omega)$. Hence, by Theorem~\ref{Minty's}, we have that such a pair $\vecu = (u,\nabla u)$ exists if $\Gamma\in QH^{1,p}_0(v,\Omega)^*$ and if $\mathcal{T}$ is a bounded, monotone, hemicontinuous, almost coercive operator. We will demonstrate each of these in turn in the following five lemmas.

\begin{lemma}\label{lem:functionalsr}
Given $1< p<\infty$, $\Gamma\in QH^{1,p}_0(v,\Omega)^*$: that is, $\Gamma$ is a bounded linear functional on $QH^{1,p}_0(v,\Omega)$.
\end{lemma}

\begin{proof}
Since the gradient, the dot product and the integral are all linear, it follows at once that $\Gamma$ is a linear functional.  Therefore, we only need to show that it is bounded.
%
But this follows from H\"older's inequality and our assumption that $f,\,g\in L^p(v,\Omega)$: arguing as we did in the proof of Lemma~\ref{Lemma5.1}, by H\"older's inequality and \eqref{eqn:subunit}  we have that for $\vecw = (w,\nabla w)\in QH_0^{1,p}(v,\Omega)$,
\begin{align*}
|\langle\Gamma,\vecw\rangle|
&\leq \int_\Omega|f|^{p-1}w\, vdx 
+ \int_\Omega |g|^{p-1}|\vect\cdot \nabla w|\, vdx \\
&\leq \|f\|^{p-1}_{L^p(v,\Omega)}\|w\|_{L^p(v,\Omega)} 
+ \|g\|^{p-1}_{L^p(v,\Omega)}\|\vect \cdot \nabla w\|_{L^p(v,\Omega)} \\
& \leq  \|f\|^{p-1}_{L^p(v,\Omega)}\|w\|_{L^p(v,\Omega)} 
+ C(\vect)\|g\|^{p-1}_{L^p(v,\Omega)}\|\nabla w \|_{QL^p(\Omega)} \\
&\leq \big[\|f\|^{p-1}_{L^p(v,\Omega)}+C(\vect)\|g\|^{p-1}_{L^p(v,\Omega)}\big]\|\vecw\|_{QH^{1,p}_0(v,\Omega)}.
\end{align*}
This completes the proof.
\end{proof}

\begin{lemma}\label{lem:boundedsr}
    Given $1< p<\infty$, the operator $\mathcal{T} : QH^{1,p}_0(v,\Omega) \to QH^{1,p}_0(v,\Omega)^*$ is bounded.
\end{lemma}

\begin{proof}
Fix $\vecu=(u,\nabla u)\in QH^{1,p}_0(v,\Omega)$, $\vecw=(w,\nabla w)\in QH^{1,p}_0(v,\Omega)$.  Then by H\"older's inequality and the assumption that $F\in L^\infty(v,\Omega)$ and is non-negative, we have that
\begin{align*}
 |\langle\mathcal{T}(\vecu),\vecw\rangle|
 &\leq \int_\Omega|\sqrt{Q}\nabla u|^{p-1}|\sqrt{Q}\nabla w|\,dx +\int_\Omega F|u|^{p-1}|w|\, v dx \\
    &\leq \|\sqrt{Q}\nabla u\|^{p-1}_{L^p(\Omega)}\|\sqrt{Q}\nabla w\|_{L^p(\Omega)}
    +\|F\|_{L^\infty(v,\Omega)}\|u\|^{p-1}_{L^p(v,\Omega)}\|w\|_{L^p(v,\Omega)}\\
    &=\|\nabla u\|^{p-1}_{QL^p(\Omega)}\|\nabla w\|_{QL^p(\Omega)}+\|F\|_{L^\infty(v,\Omega)}\|u\|^{p-1}_{L^p(v,\Omega)}\|w\|_{L^p(v,\Omega)}\\
    & \leq \big[1+\|F\|_{L^\infty(v,\Omega)}\big]\|\vecu\|^{p-1}_{QH^{1,p}_0(v,\Omega)}\|\vecw\|_{QH^{1,p}_0(v,\Omega)}.
\end{align*}
It follows at once that $\mathcal{T}$ is bounded.
\end{proof}

In the next two lemmas we prove that the operator $\mathcal{T}$ is monotone and hemicontinuous.  In these proofs, as well as in the proof of uniqueness below, we will make use of special estimates for vectors and scalars (e.g.,~\eqref{eqn:lindqvist 1}) that can be found in \cite{MR3931688}.  The interested reader can find similar estimates in \cite{MDJ2001,GM1975}  that can be used to derive ours.

\begin{lemma}\label{lem:monotonesr}
   Given $1< p<\infty$, the operator $\mathcal{T}$ is monotone.
\end{lemma}

\begin{proof}
Fix $\vecu=(u,\nabla u)\in QH^{1,p}_0(v,\Omega)$, $\vecw=(w,\nabla w)\in QH^{1,p}_0(v,\Omega)$. Then we have that
\begin{align*}
    & \langle\mathcal{T}(\vecu)-\mathcal{T}(\vecw),\vecu-\vecw\rangle \\
    &\qquad \qquad =\langle\mathcal{T}(\vecu),\vecu-\vecw\rangle-\langle\mathcal{T}(\vecw),\vecu-\vecw\rangle\\
    &\qquad \qquad =\int_\Omega|\sqrt{Q}\nabla u|^{p-2}Q\nabla u\cdot (\nabla u-\nabla w)  \,dx 
    - \int_\Omega|\sqrt{Q}\nabla w|^{p-2}Q\nabla w\cdot (\nabla u-\nabla w)  \,dx\\
    &\qquad \qquad \qquad +\int_\Omega F|u|^{p-2}u(u-w)  \,v dx - \int_\Omega F|w|^{p-2}w(u-w)  \, v dx\\
    &\qquad \qquad =\int_\Omega \big(|\sqrt{Q}\nabla u|^{p-2}Q\nabla u
    -|\sqrt{Q}\nabla w|^{p-2}Q\nabla w \big)\cdot (\nabla u-\nabla w)\,dx\\
        &\qquad \qquad \qquad +\int_\Omega F \big(|u|^{p-2} u-|w|^{p-2}w\big)(u-w) \,v dx.
 \end{align*}
 We estimate the final two integrals in turn.  To estimate the first, let $\langle\cdot,\cdot\rangle_{\R^n}$ denote the inner product on $\R^n$. For each $x\in \Omega$, the integrand of the first integral is of the form 
\begin{equation*}
    \langle|s|^{p-2}s - |r|^{p-2}r,s-r\rangle_{\R^n},
\end{equation*}
where $s,\,r\in\R^n$.  Since $p>1$, this quantity is always non-negative: see~\cite[Chapter 12, p.~74]{MR3931688}).  It follows at once that the first integral is non-negative.  We estimate the second integral in essentially the same way.  When $n=1$, we have that for all $s,\,r \in \R$,
\begin{equation*}
    \big( |s|^{p-2}s - |r|^{p-2}r\big)(s-r) \geq 0.
\end{equation*}
   Since $F$ is a non-negative function, it follows that the second integral is also non-negative.  If we combine these two estimates, we see that $\mathcal{T}$ is monotone.
\end{proof}

\begin{lemma}\label{lem:hemicontinuitysr}
    Given $1<p<\infty$, the operator $\mathcal{T}$ is hemicontinuous.
\end{lemma}

\begin{proof}
    Let $z,y\in \R$ and let $\vecu=(u,\nabla u), \vecw=(w,\nabla w)$. To simplify  notation, set $\bphi=\nabla u+z\nabla w$, $\bpsi=\nabla u+y\nabla w$, $\Phi=u+zw$ and $\Psi=u+yw$. Then we have that
\begin{align}\label{eqn 4}
    & \big|\langle\mathcal{T}(\vecu+z\vecw)-\mathcal{T}(\vecu+y\vecw), \vecw\rangle\big| \notag \\
    & \qquad \qquad =\big|\langle\mathcal{T}(\vecu+z\vecw), \vecw\rangle-\langle\mathcal{T}(\vecu+y\vecw), \vecw\rangle \big| \notag\\
    &\qquad \qquad  \leq \bigg|\int_\Omega\big[\big|\sqrt{Q}(\nabla u+z\nabla w)\big|^{p-2} Q(\nabla u+z\nabla w) \cdot\nabla w\notag \\
    &\qquad \qquad \qquad - \big|\sqrt{Q}(\nabla u+y\nabla w)\big|^{p-2} Q(\nabla u+y\nabla w))\cdot \nabla w \big] \, dx \bigg| \notag\\
    &\qquad \qquad \qquad+ \bigg|\int_\Omega\big[ F|u+zw|^{p-2}(u+zw)\cdot w
   - F|u+yw|^{p-2} (u+yw)\cdot w\big]\, v dx\bigg| \notag\\
    &\qquad \qquad =\bigg| \int_\Omega \big[\big|\sqrt{Q}\bphi\big|^{p-2}  \sqrt{Q}\bphi \cdot \sqrt{Q}\nabla w
    - \big|\sqrt{Q}\bpsi\big|^{p-2}\sqrt{Q}\bpsi \cdot \sqrt{Q}\nabla w\big]\,dx\bigg| \notag\\
    &\qquad \qquad \qquad +\bigg| \int_\Omega \big[ F|\Phi|^{p-2} \Phi\cdot w - F|\Psi|^{p-2} \Psi\cdot w\big] \,v dx\bigg| \notag\\
      &\qquad \qquad  \leq  \int_\Omega \big||\sqrt{Q}\bphi|^{p-2}\sqrt{Q}\bphi 
      - |\sqrt{Q}\bpsi|^{p-2}\sqrt{Q}\bpsi\big| |\sqrt{Q}\nabla w|\,dx  \\
    &\qquad \qquad \qquad +\|F\|_{L^\infty(v,\Omega)}\int_\Omega \big||\Phi|^{p-2}\Phi - |\Psi|^{p-2}\Psi\big||w| \,v dx.\notag
\end{align}

To estimate~\eqref{eqn 4} we will consider two cases: $p\ge 2$ and $1<p<2$. If $p\ge2$, then by (\cite[Chapter 10, p.~99]{MR3931688}) we have that  for $r,s\in \R^n$,
\begin{equation}\label{eqn:lindqvist 1}
    \big||r|^{p-2}r-|s|^{p-2}s\big|\leq (p-1)|r-s|\big(|s|^{p-2}+|r|^{p-2}\big).
\end{equation}
The same inequality holds if $r,\,s$ are scalars. (A slightly different version of this inequality is proved, but the proof can be easily modified to prove~\eqref{eqn:lindqvist 1}.)  
Furthermore, if we let $r=\sqrt{Q}\bphi$ and $s= \sqrt{Q}\bpsi$, then we immediately have that $r-s= (z-y)\sqrt{Q} \nabla w$.  
Therefore, we can  apply~\eqref{eqn:lindqvist 1} 
to~\eqref{eqn 4} to get
\begin{align}\label{ineq 7}
   &  |\langle\mathcal{T}(\vecu+z\vecw)-\mathcal{T}(\vecu+y\vecw), \vecw\rangle| \notag \\
      &\qquad \qquad \leq (p-1)|z-y|\int_\Omega\big|\sqrt{Q}\nabla w\big|^2\big[|\sqrt{Q}\bpsi|^{p-2}+|\sqrt{Q}\bphi|^{p-2}\big]\,dx\\
    &\qquad \qquad \qquad+(p-1)|z-y|\|F\|_{L^\infty(v,\Omega)}\int_\Omega|w|^2\big[|\Psi|^{p-2}+|\Phi|^{p-2}\big] \,v dx.\notag
\end{align}

If $p=2$, then both integrals on the right-hand side are finite, so the right-hand side tends to $0$ as $z\rightarrow y$. Hence, $\mathcal T$ is hemicontinuous.  If $p>2$, we estimate each of these integrals in turn.  Using Hölder's inequality with exponents $\frac{p}{2}$ and $\frac{p}{p-2}$ we get
%
\begin{multline*}
\int_\Omega\big|\sqrt{Q}\nabla w\big|^2\big[|\sqrt{Q}\bpsi|^{p-2}+|\sqrt{Q}\bphi|^{p-2}\big]\,dx
\leq \|\nabla w\|^2_{QL^p(\Omega)} \||\sqrt{Q}\bpsi|^{p-2}+|\sqrt{Q}\bphi|^{p-2} \|_{L^{\frac{p}{p-2}}(\Omega)} \\
 \leq \|\nabla w\|^2_{QL^p(\Omega)}\big(\|\bpsi\|^{p-2}_{QL^p(\Omega)}+\|\bphi\|^{p-2}_{QL^p(\Omega)}\big).
\end{multline*}
Similarly, we have that 
\begin{equation*} \int_\Omega|w|^2\big[|\Psi|^{p-2}+|\Phi|^{p-2}\big]\, vdx
\leq\|w\|^2_{L^p(v,\Omega)}\big(\|\Psi\|^{p-2}_{L^p(v, \Omega)}+\|\Phi\|^{p-2}_{L^p(v,\Omega)}\big).\\
\end{equation*}
It follows from their definitions that $\bphi,\,\bpsi \in QL^p(\Omega)$ and $\Phi,\,\Psi \in L^p(v,\Omega)$.  Thus, the integrals 
on the right-hand side of~\eqref{ineq 7} are finite, and so the right-hand side  converges to $0$ as $z\rightarrow y$.  Therefore, $\mathcal T$ is hemicontinuous for $p\geq 2$.

\medskip

Now suppose $1<p<2$. Then again from (\cite[chapter 12, p.~98]{MR3931688}) we have that for $r,s\in \R^n$,
\begin{equation*}
    \big||s|^{p-2}s-|r|^{p-2}r\big|\leq 2^{2-p}|s-r|^{p-1};
\end{equation*}
again, the same inequality holds if $r,\,s$ are scalars.  Therefore, if we apply this inequality to~\eqref{eqn 4} and then use H\"older's inequality with exponents $p$ and $p'$, we get
\begin{align*}
     |\langle\mathcal{T}(\vecu+z\vecw)-\mathcal{T}(\vecu+y\vecw), \vecw\rangle|&\leq \bigg|\int_\Omega\sqrt{Q}\nabla w\cdot\big[|\bpsi|^{p-2}\bpsi-|\bphi|^{p-2}\bphi\big]\,dx\bigg|\\
     &\qquad \qquad +\|F\|_{L^\infty(v,\Omega)} \bigg|\int_\Omega w\cdot\big[|\Psi|^{p-2}\Psi-|\Phi|^{p-2}\Phi\big]v\,dx\bigg|\\
     &\leq \int_\Omega \big|\sqrt{Q}\nabla w\big| \big||\bpsi|^{p-2}\bpsi-|\bphi|^{p-2}
     \bphi\big|\,dx\\
     &\qquad \qquad+\|F\|_{L^\infty(v,\Omega)} \int_\Omega |w| \big||\Psi|^{p-2}\Psi-|\Phi|^{p-2}\Phi\big|v\,dx\\
     &\leq C \int_\Omega\big|\sqrt{Q}\nabla w\big||\bpsi-\bphi|^{p-1}\,dx\\
     &\qquad \qquad+ C \|F\|_{L^\infty(v,\Omega)}\int_\Omega|w||\Psi-\Phi|^{p-1}\,dx.\\
  & \leq C\|\nabla w\|_{QL^p(\Omega)}\|\bpsi-\bphi\|^{p-1}_{L^p(\Omega)} \\
  & \qquad \qquad +C\|F\|_{L^\infty(v,\Omega)}\|w\|_{L^p(v,\Omega)}\|\Psi-\Phi\|^{p-1}_{L^p(\Omega)} \\ 
  &=C|z-y|^{p-1}\|\nabla w\|^{p-1}_{QL^p(\Omega)} \\ 
  & \qquad +C|z-y|^{p-1}\|F\|_{L^\infty(v,\Omega)}\|w\|^{p-1}_{L^p(v,\Omega)}.
\end{align*}
The final term tends to zero as $z\rightarrow y$, so $\mathcal{T}$ is hemicontinuous when $1<p<2$.
\end{proof}

\begin{lemma} \label{lemma:coercive}
    Given $1<p<\infty$, if $(v,Q)$ has the $p$-Sobolev property on $\Omega$, then the operator $\mathcal{T}$ is almost coercive.
\end{lemma}

\begin{remark}
    We note that in the proof of the existence of solutions to the Dirichlet problem~\eqref{Dirichlet prob.}, it is only in the proof of Lemma~\ref{lemma:coercive} that we use the $p$-Sobolev property.  We also need it to prove regularity:  see below.
\end{remark}

\begin{proof}
    We will show that there exists a $\lambda>0$ sufficiently large, depending on the data functions $f$ and $g$ used to define $\Gamma$, such that  for any $\vecu \in QH^{1,p}_0(v,\Omega)$ with $\|\vecu\|_{QH^{1,p}_0(v,\Omega)}>\lambda$, $\langle\mathcal{T}(\vecu),\vecu\rangle>\langle\Gamma,\vecu\rangle$. Fix  $\vecu=(u,\nabla u)\in QH^{1,p}_0(v,\Omega)$ such that $\|\vecu\|_{QH^{1,p}_0(v,\Omega)}>0$.   By Lemma~\ref{lemma:approximation} we can apply the Sobolev inequality~\eqref{Sobolev prop} to $(u,\nabla u)$ while also testing $\vecu$ against itself in the definition of weak solution to get
\begin{multline*}
    \|u\|^p_{L^p(v,\Omega)}
    \leq S(p,1)^p\|\nabla u\|^p_{QL^p(\Omega)}\\
    \leq S(p,1)^p\left(\|\nabla u\|^p_{QL^p(\Omega)}+\int_\Omega F|u|^p\,vdx\right)
    =S(p,1)^p\langle\mathcal{T}(\vecu),\vecu\rangle;
\end{multline*}
the second inequality holds since $F$ is non-negative.  It follows at once from this estimate that
\begin{equation} \label{eqn:coercive-bound}
    \|\vecu\|^p_{QH^{1,p}_0(v,\Omega)}\leq (S(p,1)^p+1)\langle\mathcal{T}(\vecu),\vecu\rangle.
\end{equation}
Since $p>1$, by Hölder's inequality, \eqref{eqn:subunit}, and \eqref{eqn:coercive-bound} we have
\begin{align*}
    |\langle\Gamma,\vecu\rangle|
    =\bigg|\int_\Omega|f|^{p-2}f&u\,vdx+ \int_\Omega |g|^{p-2}g \, \vect \cdot \nabla u\, vdx\bigg|\\
    &\leq \Big(\|f\|^{p-1}_{L^p(v,\Omega)}+C(\vect)\|g\|^{p-1}_{L^p(v,\Omega)}\Big)\|\vecu\|_{QH^{1,p}_0(v,\Omega)}\\
    &\leq (S(p,1)^p+1)\Big(\|f\|^{p-1}_{L^p(v,\Omega)}+C(\vect)\|g\|^{p-1}_{L^p(v,\Omega)}\Big)\|\vecu\|^{1-p}_{QH^{1,p}_0(v,\Omega)}
    \langle\mathcal{T}(\vecu),\vecu\rangle.
\end{align*}

We now consider two cases.  If  $f\not\equiv0$ or $g\not\equiv 0$, then $|\langle\Gamma,\vecu\rangle|<\langle\mathcal{T}(\vecu),\vecu\rangle$ provided that
\begin{equation*}
    (S(p,1)^p+1)\Big(\|f\|^{p-1}_{L^p(v,\Omega)}+C(\vect)\|g\|^{p-1}_{L^p(v,\Omega)}\Big)\|\vecu\|^{1-p}_{QH^{1,p}_0(v,\Omega)}<1, 
\end{equation*}
or, equivalently, 
\[ \|\vecu\|_{QH^{1,p}_0(v,\Omega)} > (S(p,1)^p+1)^{\frac{1}{p-1}}\Big(\|f\|^{p-1}_{L^p(v,\Omega)}+C(\vect)\|g\|^{p-1}_{L^p(v,\Omega)} \Big)^\frac{1}{p-1}.  \]
On the other hand, if $f\equiv 0\equiv g$, then $\Gamma=0\in(QH^{1,p}_0(v,\Omega))^*$. Hence, if we let $\lambda=1$, then by inequality~\eqref{eqn:coercive-bound},
\[ |\langle\Gamma,\vecu\rangle|= 0 <\langle\mathcal{T}(\vecu),\vecu\rangle. \]
Thus,  for any $f\in L^p(v,\Omega)$, $\mathcal{T}$ is almost coercive with 
\begin{equation*}
    \lambda=\max\Big\{1,(S(p,1)^p+1)^{\frac{1}{p-1}}\Big(\|f\|^{p-1}_{L^p(v,\Omega)}+C(\vect)\|g\|^{p-1}_{L^p(v,\Omega)}\Big)^\frac{1}{p-1}\Big\}.
\end{equation*}
\end{proof}

It follows from the above lemmas that there exists a degenerate weak solution $(u,\nabla u)\in QH^{1,p}_0(v,\Omega)$ of the Dirichlet problem~\eqref{Dirichlet prob.}.  We now show that the solutions satisfy~\eqref{regularity}.  To do so, we argue as in the proof of Lemma~\ref{Lemma5.1}, using $(u,\nabla u)$ as a test function.  Then from inequality~\eqref{eqn:caccioppoli2} we have that
\begin{align*}
    \|\nabla u\|_{QL^p(\Omega)}^p
    & \leq \|f\|_{L^p(v,\Omega)}^{p-1} \|u\|_{L^p(v,\Omega)} 
    + C(\vect)  \|g\|_{L^p(v,\Omega)}^{p-1} \|\nabla u\|_{QL^p(\Omega)}. \\
    \intertext{By~\eqref{Sobolev prop} and Young's inequality}
    & \leq \big[S(p,1) \|f\|_{L^p(v,\Omega)}^{p-1} 
    + C(\vect)  \|g\|_{L^p(v,\Omega)}^{p-1}  \big] \|\nabla u\|_{QL^p(\Omega)} \\
    & \leq 2^{\frac{p'}{p}}\big[S(p,1) \|f\|_{L^p(v,\Omega)}^{p-1} 
    + C(\vect)  \|g\|_{L^p(v,\Omega)}^{p-1}  \big]^{p'}
    + \tfrac{1}{2}\|\nabla u\|_{QL^p(\Omega)}^p.
\end{align*}
If we rearrange terms and take the $p$-th root of both sides, we get
\begin{equation} \label{eqn:grad-regularity2} 
\|\nabla u\|_{QL^p(\Omega)} 
\leq C\big[ \|f\|_{L^p(v,\Omega)}^{p-1}  + \|g\|_{L^p(v,\Omega)}^{p-1} ]^{\frac{1}{p-1}}. 
\end{equation}
The Caccioppoli-type inequality~\eqref{eqn:grad-regularity} follows immediately.  Moreover, by the Sobolev inequality~\eqref{Sobolev prop}, we get inequality~\eqref{regularity}. 

\section{Uniqueness of solutions}
\label{section:unique}

In this section we complete the proof of Theorem~\ref{equivalence} by showing  that if $(v,Q)$ has the $p$-Sobolev and $p$-Dirichlet properties, then any solution of the Dirichlet problem~\eqref{Dirichlet prob.} is unique.  To prove this we adapt an argument from~\cite{Yaremenko}.

Suppose to the contrary that solutions are not unique:  that for some coefficients and data, there exist two degenerate weak solutions $(u_1,\nabla u_1),\, (u_2,\nabla u_2) \in QH^{1,p}_0(v,\Omega)$. Define $(w,\nabla w) =(u_1-u_2, \nabla u_1-\nabla u_2) \in QH^{1,p}_0(v,\Omega)$.  (Note that by Remark~\ref{remark:grad-linear}, $\nabla(u_1-u_2)=\nabla u_1-\nabla u_2$.)    By Lemma~\ref{test func.}, we can use $(w, \nabla w)$ as a test function in the definition of a degenerate weak solution \eqref{weak soln.}. If we do this for both $(u_1,\nabla u_1)$ and $(u_2,\nabla u_2)$ and take the difference of the two equations, we get
\begin{multline}\label{eqn3}
    \int_\Omega\big(|\sqrt{Q}\nabla u_1|^{p-2}\sqrt{Q}\nabla u_1-|\sqrt{Q}\nabla u_2|^{p-2}\sqrt{Q}\nabla u_2\big)\cdot\sqrt{Q}\nabla(u_1-u_2)\, dx \\
    + \int_\Omega F\big(|u_1|^{p-2}u_1 - |u_2|^{p-2}u_2\big) (u_1-u_2) \,v dx
=0.
\end{multline}
We now apply the following inequality (see~\cite[Chapter~12, p.~100]{MR3931688}):  for $1< p<\infty$ and all $a,\,b \in \R^n$,
\begin{equation*}
    (|b|^{p-2}b-|a|^{p-2}a)\cdot(b-a)\geq 0,
\end{equation*}
and the inequality is strict  if $a\neq b$.   The corresponding inequality for scalars is also true.  Therefore, the integrands on the left-hand side of~\eqref{eqn3} are both non-negative.  Hence,  if $\sqrt{Q}\nabla u_1(x) \neq \sqrt{Q}\nabla u_2(x)$ on a set of positive measure, then the left-hand side of~\eqref{eqn3} is positive, which contradicts the fact that it is equal to $0$.  Therefore, $\sqrt{Q}\nabla u_1 = \sqrt{Q}\nabla u_2$ a.e.  

A similar argument shows $u_1=u_2$ $v$-almost everywhere, but only if we assume that $F$ is positive on the set where $u_1 \neq u_2$.  To avoid this, we instead use the Sobolev inequality.   By inequality~\eqref{Sobolev prop} we have that
\begin{equation*}
        \int_\Omega|u_1-u_2|^{p} \,vdx\leq S(p,1)^p\int_\Omega\big|\sqrt{Q}\nabla (u_1-u_2)\big|^p \,dx= 0.   
\end{equation*}
 Hence,  $u_1-u_2=0$  $v$-a.e. Thus, $(u_1,\nabla u_1) = (u_2, \nabla u_2)$ in $QH_0^{1,p}(v,\Omega)$, and so the solution to the Dirichlet problem \eqref{Dirichlet prob.} is unique.

\bibliographystyle{plain}
\bibliography{bibliography}

\end{document}